\documentclass[12pt]{amsart}
\usepackage{amsmath,amssymb,amsthm,amsfonts,enumerate,array,color,lscape,fancyhdr,layout,pst-all}
\usepackage{xcolor}
\usepackage[english]{babel}
\usepackage{young}
\usepackage{float}
\usepackage[all]{xy}
\usepackage[normalem]{ulem}
\usepackage[active]{srcltx}
\usepackage{mathrsfs}
\usepackage{etoolbox}
\usepackage[margin=1in]{geometry}
\usepackage[colorlinks=true, breaklinks=true, allcolors=blue]{hyperref}

\usepackage{xspace}
\newlength\yStones
\newlength\xStones
\newlength\xxStones

\makeatletter
\def\Stones{\pst@object{Stones}}
\def\Stones@i#1{%
  \pst@killglue%
  \begingroup%
  \use@par%
  \setlength\xxStones{\xStones}%
  \expandafter\Stones@ii#1,,\@nil
  \endgroup
  \global\addtolength\xStones{0.6cm}%
  \global\addtolength\yStones{-7.5mm}}%
\def\Stones@ii#1,#2,#3\@nil{%
  \rput(\xxStones,\yStones){%
    \psframebox[framesep=0]{%
      \parbox[c][6mm][c]{11mm}{\makebox[11mm]{$#1$}}}}%
  \addtolength\xxStones{1.2cm}%
  \ifx\relax#2\relax\else\Stones@ii#2,#3\@nil\fi}
\makeatother

\def\Stone#1{\fbox{\makebox[12mm]{\strut#1}}\kern2pt}

\allowdisplaybreaks
                
\theoremstyle{plain}
\newtheorem{theorem}{Theorem}[section]
\newtheorem{lemma}[theorem]{Lemma}
\newtheorem{corollary}[theorem]{Corollary}
\newtheorem{proposition}[theorem]{Proposition}

\theoremstyle{definition}

\newcommand{\Z}{\mathbb{Z}} 
\newcommand{\bk}{\Bbbk}
\newcommand{\A}{\mathbb{A}} 

\newcommand{\fm}{\mathfrak{m}}

\newcommand{\V}{\mathcal{V}}
\newcommand{\Spec}{\mathrm{Spec} \,}
\newcommand{\cO}{\mathcal{O}}

\newcommand{\boh}{\overline{0}}
\newcommand{\bon}{\overline{1}}

\newcommand{\Der}{\mathrm{Der}}

\newcommand{\ad}{\mathrm{ad}}

\newcommand{\cT}{\mathcal{T}}

\newcommand{\cV}{\mathcal{V}}

\begin{document}

\begin{title}{Simplicity of the Lie superalgebra of vector fields}
\end{title}

\author[H. Rocha]{Henrique Rocha}


\begin{abstract}
For a finitely generated integral super domain $A$, we prove the Lie superalgebra $\cV = \Der(A)$ of super derivations is a simple Lie superalgebra.
\end{abstract}




\maketitle




\section*{Introduction}
Jordan \cite{Jor86} proved the Lie algebra of vector fields on a algebraic variety $X$ is simple if $X$ is smooth. A natural question that arise is if this is result is true for an affine supervariety. On this paper, we prove that for a finitely generated integral super domain $A$, the Lie superalgebra $\cV = \Der(A)$ of super derivations is a simple Lie superalgebra. We use a similar approach to \cite{BF18} to prove it.

In Section \ref{section:preliminars}, we give general definitions. In Section \ref{section:mainresult} we prove the main result of this paper.

\section{Preliminaries}\label{section:preliminars}

Let $\bk$ be an algebraically closed field of characteristic $0$. Unless otherwise stated, all (super)algebras, (super)schemes, (super) vector spaces will be taken over $\bk$. For general definitions and results, we refer to \cite{CCF11}. For a superalgebra $A = A_{\boh} \oplus A_{\bon}$ we denote $\overline{A} = A / (A_{\bon})$ the quotient of $A$ by the ideal generated by its odd elements. Except as otherwise specified, we will use Latin letters, like $x,t, \dots$, for even elements, and Greek letters, like $\xi, \theta, \dots$, for odd elements.

Let $A$ be a finitely generated integral superdomain, i.e. $A \setminus A_{\bon} $ does not contain zero divisors. Then, $(A_{\bon})$ is a prime ideal, the localization map $A \rightarrow A_{(A_{\bon})}$ is injective and $\overline{A}$ is an integral domain.  By \cite{VMP90}, there exists a finitely generated projective $\overline{A}$-module $N$ such that $A \cong \Lambda^{\bullet}_{\overline{A}} N$. We denote by $X = \underline{\Spec}(A) = (|X|, \cO)$ the structure superscheme of $A$ with structure sheaf $\cO$, then $X$ is irreducible and integral. For a closed point $p \in X$, we denote by $\mathfrak{m}_p$ the maximal ideal of the local algebra $\cO_p$. Since $A$ is an integral superdomain, if the reduced scheme $X^r = \Spec(\overline{A})$ is smooth, then $X$ is smooth as well. For this and other general results about smooth supervarieties, we refer to \cite{She21} and references therein.

We say an even element $x \in A_{\bon}$ is $A$-regular if the associated multiplication map is injective. An odd element $\xi \in A_{\bon}$ is called $A$-regular if the sequence 
\[
A \xrightarrow{\xi} A \xrightarrow{\xi} A 
\]
is exact. We say a sequence $a_1,\dots,a_r \in A$ of homogeneous elements is regular if $a_i $ is $A/(a_1,\dots,a_{i-1})$-regular for each $i=1,\dots,r$.

\section{Simplicity of Lie superalgebra}\label{section:mainresult}
Functions $t_1,\dots,t_r,\theta_1,\dots,\theta_s \in \cO(U)$ define  a \emph{coordinate system} in a open subset $U$ of $X$ if they are regular on $U$ and the sections $dt_1,\dots,dt_r,d\theta_1,\dots,d\theta_s$ trivializes the cotangent bundle on $U$. In this case, the derivations $\partial_{t_1},\dots,\partial_{t_r},\partial_{\theta_1},\dots, \partial_{\theta_s} $ are well-define on $\cO(U)$ and \[ \cT (U) = \bigoplus_{i=1}^r \cO(U)\partial_{t_i} \oplus\bigoplus_{j=1}^s \cO(U) \partial_{\theta_j} .\]

\begin{proposition}
    If $X$ is smooth with dimension $r|s$, then for each closed point $p \in X$ there exists an affine open neighboorhod $U$ of $p$ and global functions $t_1,\dots,t_r,\theta_1,\dots,\theta_s \in A$ that form a coordinate system on $U$.
\end{proposition}
\begin{proof}
    Assume $X \subset \A^{m|n}$. Then $A \cong \bk[x_1,\dots,x_m,\xi_1,\dots,\xi_n]/I$ for some ideal $I$, then $t_i = x_i - x_i (p)$, $\theta_j = \xi_i$ generates the corresponding maximal ideal $\mathfrak{m}$ of $A$. Since $X$ is smooth, there exists $t_{i_1},\dots,t_{i_r},\theta_{j_1},\dots,\theta_{j_s}$ such that the image of the respective local sections $t_{i_1},\dots,t_{i_r},\theta_{j_1},\dots,\theta_{j_s} \in \mathfrak{m}_p$ are regular and form a basis of the tangent space $T_pX$. So there exists an open set $U$ that contains $p$ such that $t_{i_1},\dots,t_{i_r},\theta_{j_1},\dots,\theta_{j_s}$ form a coordinate system on $U$. Note that $t_{i_1},\dots,t_{i_r},\theta_{j_1},\dots,\theta_{j_s}$ generates the maximal ideal $\mathfrak{m}_p$ of $\cO_p$. 
\end{proof}

Consider the $\fm_p$-adic filtration defined on $\cO_p$ and denote by $\widehat{\cO_p} = \lim_{\leftarrow} A/\fm_p^i$ its completion. Fix $t_1,\dots,t_r,\theta_1,\dots,\theta_s\in A$ and an affine neighbourhood $U$ of $p$ such that $t_1$, $\dots$, $t_r$, $\theta_1$, $\dots$, $\theta_s$ is a coordinate system on $U$. We may assume sections $t_1,\dots,t_r,\theta_1,\dots,\theta_s\in \cO_p$ generates the maximal ideal $\mathfrak{m}_p$.

\begin{proposition}\label{proposition:descriptionofA}
There is an isomorphism $\widehat{\cO_p} \rightarrow \bk [[T_1,\dots,T_r,\Theta_1,\dots,\Theta_s]]$ such that $t_i \mapsto T_i$ and $\theta_j \mapsto \Theta_j$. Furthermore, $\cO_p \cong \overline{\cO_p} [\Theta_1,\dots,\Theta_s]$ where $\overline{\cO_p}= \cO_p/(\mathcal{O}_{p,\bon})$ and $\theta_i \mapsto \Theta_i$.
\end{proposition}
\begin{proof}
It follows from \cite[Theorem 8.3]{She21}.
\end{proof}

Since $A$ is a super domain, the localization map $A \rightarrow \cO_p$ is an embedding. Thus, by Proposition \ref{proposition:descriptionofA}, there is an embedding \[ \pi: A \hookrightarrow \bk[[T_1,\dots,T_r,\Theta_1,\dots,\Theta_s]]\] such that $ \pi(t_i) = T_i$ and $\pi(\theta_j) = \Theta_j$. Let $\widehat{\mathfrak{L}} = \Der ( \bk[[T_1,\dots,T_r|\Theta_1,\dots,\Theta_s]] ) $, then \[\left \{\partial_{T_1},\dots,\partial_{T_r},\partial_{\Theta_1},\dots,\partial_{\Theta_s}   \right\}\]
is a basis of $\widehat{\mathfrak{L}}$ as a $\bk[[T_1,\dots,T_r,\Theta_1,\dots,\Theta_s]]$-module. 
\begin{proposition}\label{proposition:pihatmorphism001}
There exists a unique embedding $\widehat{\pi}: \cV \rightarrow \widehat{\mathfrak{L}}$ such that $\widehat{\pi} (\eta) \pi (f) = \pi(\eta(f))$ for all $f \in A$ and $\eta \in \V$.
\end{proposition}
\begin{proof}
Fix $\eta \in \cV$. There is a unique way to define $\widehat{\pi}$ which is \[\widehat{\pi}(\eta) = \sum_{j=1}^r \pi(\eta(t_i)) \partial_{T_i} + \sum_{j=1}^s \pi(\eta(\theta_j)) \partial_{\Theta_j}
\]
If $\mathfrak{m}$ is the maximal ideal of $A$ associated to $p$, then we may consider the $\mathfrak{m}$-adic topology on $A$. Similarly, the maximal ideal of $\bk [[T_1,\dots,T_r,\Theta_1,\dots,\Theta_s]]$ defines a topology on it. Both topologies are separable, $\bk [t_1,\dots,t_r,\theta_1,\dots,\theta_s] \subset A$ is dense on $A$ and $\pi : A \rightarrow \bk[[T_1,\dots,T_r,\Theta_1,\dots,\Theta_s]]$ is a continuous map. Furthermore, every derivation in $\cV$ and $ \widehat{\mathfrak{L}}$ are continuous on this topologies.  For each element $f \in \bk [t_1,\dots,t_r,\theta_1,\dots,\theta_s] \subset A$, the equality $\widehat{\pi} (\eta) \pi (f) = \pi(\eta(f))$ holds. The main claim follows by the continuity of the maps $\widehat{\pi} \circ \pi $ and $\pi \circ \eta$.
\end{proof}

Suppose $U = D(h')$, then there exists $h = h'^S$ for some natural number $S$ such that $ \tau_i = h\partial_{t_i}$ and $\sigma_j = h \partial_{t_j}$ are global sections. Furthermore, 
\[
  \widehat{\pi} (\tau_k)= \pi(h) \partial_{T_k}, \quad
   \widehat{\sigma_l} (\sigma_l)  = \pi(h) \partial_{\Theta_l},
\]
for each $k=1,\dots,r$, $l=1,\dots,s$.

The Lie superalgebra $\widehat{\mathfrak{L}}$ has a filtration given by
\[
\widehat{\mathfrak{L}} = \widehat{\mathfrak{L}} (-1) \supset \widehat{\mathfrak{L}} (0) \supset \widehat{\mathfrak{L}} (1)  \supset \widehat{\mathfrak{L}} (2)  \supset \cdots
\]
where $\widehat{\mathfrak{L}}(i-1) = \fm^{i}_0 \widehat{\mathfrak{L}}$ and $\fm_0$ is the ideal of $\bk[[T_1,\dots,T_r | \Theta_1,\dots, \Theta_s ]]$ generated by $T_1$, $\dots$, $T_r$, $\Theta_1$, $\dots$, $\Theta_s$. For all $i,j$, we have $[ \widehat{\mathfrak{L}}(i) , \widehat{\mathfrak{L}}(j) ] \subset \widehat{\mathfrak{L}}(i+j)$ for $i+j \geq -1$. Consider the associated graded Lie superalgebra 
\[
\textup{Gr} \, \widehat{\mathfrak{L}} = \widehat{\mathfrak{L}}(-1) / \widehat{\mathfrak{L}}(0) \oplus \widehat{\mathfrak{L}}(0) / \widehat{\mathfrak{L}}(1) \oplus \widehat{\mathfrak{L}}(1) / \widehat{\mathfrak{L}}(2) \oplus \cdots,
\]
then $\textup{Gr} \, \widehat{\mathfrak{L}}$ is isomorphic to $\bk[T1,\dots,T_r,\Theta_1,\cdots,\Theta_s]$. Define $\omega: \widehat{\mathfrak{L}} \rightarrow \textup{Gr} \, \widehat{\mathfrak{L}}$ by $\omega(\mu) = \mu + \widehat{\mathfrak{L}}(j+1) \in \widehat{\mathfrak{L}}(j) / \widehat{\mathfrak{L}}(j+1) $ for $\mu \in \widehat{\mathfrak{L}}(j) \setminus \widehat{\mathfrak{L}}(j+1)$. Thus, $\omega$ is not linear, but if $[\omega(\eta),\omega(\mu) ] \neq 0 $ then $\omega([\eta,\mu]) = [\omega(\eta),\omega(\mu) ]$. Note that $\omega(\tau_i)$ is a non-zero multiple of $\partial_{t_i} $ and $\omega(\sigma_i)$ is a non-zero multiple of $\partial_{\theta_j} $.

\begin{proposition}\label{proposition:thereexistsmufwithmufpzero}
Suppose $r|s \geq 1|0$. Let $\mathcal{J}$ be a nonzero $\Z_2$-graded ideal of $\V$, $p$ be a smooth closed point of $X$ and a $t_1,\dots,t_r,\theta_1,\dots,\theta_s$ coordinate system as before. Then 
\begin{enumerate}
    \item There exists $y \in \{t_1,\dots,t_r,\theta_1,\dots,\theta_s \}$ and $\mu \in \mathcal{J}_{|y|}$ such that $\mu(y)(p) \neq 0$.
    \item For all $y \in \{t_1,\dots,t_r,\theta_1,\dots,\theta_s \}$, there exists $\mu \in \mathcal{J}_{|y|}$ such that $\mu(y)(p) \neq 0$. In particular, $\mathcal{J}_{\boh} \neq 0$.
    \item There exists $g \in A_{\boh}$ and $\mu \in \mathcal{J}_{\boh}$ such that $\mu(\mu(g))(p) \neq 0$.
\end{enumerate}
\end{proposition}
\begin{proof}
Let $\eta \in \mathcal{J}$. Since $\mathcal{J}$ is $\Z_2$-graded, we may assume $\eta$ homogeneous. Write
\[
\omega(\widehat{\pi}(\eta) ) = \sum_{i=1}^{r} u_i \partial_{T_i} + \sum_{j=1}^s u_{r+j}\partial_{\Theta_j} \neq 0
\]
where $u_1,\dots,u_{r+s} \in \bk[T_1,\dots,T_r,\Theta_1,\dots,\Theta_s]$ are homogeneous polynomials with the same degree. Choose a nonzero polynomial $u_{j_0} \in \{u_1,\dots,u_{r+s}\} $, and let $Y  = T_{j_0}$ if $j_0 \geq r$ or $Y = \Theta_{j_0-r}$ if $j_0 > r$. Let $T_1^{k_1} \cdots T_r^{k_r} \Theta_{i_1} \cdots \Theta_{i_l}$ be a monomial that occurs in $u_{j_0}$, then
\[
0 \neq \left (\partial_{T_1} \right)^{k_1}  \cdots \left (\partial_{T_1} \right)^{k_s} \left (\partial_{\Theta_{i_1}}   \right ) \cdots \left (\partial_{\Theta_{i_l}}   \right ) u_{j_0} \in \bk 
\]

Consider 
\[
\mu = \ad(\tau_1)^{k_1} \cdots \ad(\tau_r)^{k_r} \ad(\sigma_{i_1} ) \cdots \ad(\sigma_{i_l}) \eta  \in \mathcal{J}.
\]
Then, $\mu$ is homogeneous and it has the same parity of $y$. Since $\omega$ preserves brackets, we see that
\[\omega(\widehat{\pi}(\mu)) = \ad (\omega(\widehat{\pi}(\tau_1)))^{k_1} \cdots \ad(\omega(\widehat{\pi}(\tau_s)))^{k_s}\ad(\omega(\widehat{\pi} (\sigma_{i_1})) )\cdots \ad(\omega(\widehat{\pi}(\sigma_{i_l} )) )  \omega(\eta) \in  \widehat{\mathfrak{L}}(-1) \setminus \widehat{\mathfrak{L}}(0) 
\]
and it contains $\frac{\partial }{\partial Y}$. Hence, $\omega(\widehat{\pi} ( \mu) ) Y \neq 0$, which implies that $\mu(y)(p) \neq 0$ where $y \in \{t_1,\dots,t_r,\theta_1,\dots,\theta_s \}$ is the coordinate such that $\widehat{\pi}(y) = Y $.

Let $i=1,\dots,r$, then $ \omega(  \widehat{\pi} ( [y \tau_i, \mu] ) ) = [\omega(  \widehat{\pi} ( y \tau_i )), \omega(  \widehat{\pi}( \mu)) ]$ contains $\frac{\partial}{\partial T_i}$. Hence, $  \omega(  \widehat{\pi} ( [y \tau_i, \mu] ) )T_i  \neq 0$, thus $[y \tau_i, \mu] (t_i)(p) \neq 0$. Using the same proof but taking $[y \sigma_j, \mu]$, it is possible to show that $[y\sigma_j,\mu ](\theta_j) (p) \neq 0$.

For the last item, take $\mu' =[y \tau_i, \mu] \in \mathcal{J}_{\boh}$. If $\mu'(\mu'(t_i))(p) \neq 0 $, it is done. If $\mu'(\mu'(t_i))(p) =0 $, then $t_i^2 \neq 0$ and
\[
\mu' (\mu'(t_i^2 ))(p)  = 2 \mu'(t_i\mu'(t_i))(p) = 2 \mu'(t_i) (p) \mu'(t_i)(p)+ 2 t_i \mu'(\mu'(t_i'))  = 2  (\mu'(t_i) (p) )^2 \neq 0 .
\]
\end{proof}
\begin{corollary}\label{corollary:thismapidealsurjective}
If $\mathcal{J}$ is an nonzero ideal of $\V$ and $p$ is a smooth closed point, then $\mathcal{J} \rightarrow T_p X$ given by $D \mapsto (f \mapsto D(f)(p) )$ is surjective.
\end{corollary}
\begin{proof}
The image of the vector fields $[y\tau_i, \mu ], \ [y \sigma_j ,\mu]  \in \mathcal{J}$, $i=1,\dots,r$, $j=1,\dots,s$, span $T_p X$.
\end{proof}

\begin{lemma}\label{lemma:calculationsboring}
 Let $\mathcal{J}$ be an nonzero ideal of $\V$, then
\begin{enumerate}
    \item\label{item:001calculationsboring} If $\mu \in \mathcal{J}_{\boh}$, then $\mu(f) \mu \in \mathcal{J}$ for all $f \in A$.
    \item\label{item:002calculationsboring} If $g \in A_{\boh}$ and $\mu \in \mathcal{J}_{\boh}$, then $\mu(f)\mu(g) \mu \in \mathcal{J}$ for all $f \in A$.
    \item\label{item:003calculationsboring} If $g \in A_{\boh}$ and $\mu \in \mathcal{J}_{\boh}$, then $f \mu(\mu(g)) \mu \in \mathcal{J}$ for all $f \in A$.
\end{enumerate}
\end{lemma}
\begin{proof}
Let $\mu \in \mathcal{J}_{\boh}$ and $f \in A$, then  $[\mu,f\mu] = \mu(f) \mu +(-1)^{|f||\mu|} f[\mu,\mu]=\mu(f)\mu \in \mathcal{J}$ and item \eqref{item:001calculationsboring} follows. For each $g\in A_{\boh}$, $[\mu,f\mu(g)\mu] - [f\mu,\mu(g)\mu] = 2\mu(g)\mu(f) \in \mathcal{J}$, hence \eqref{item:001calculationsboring} is proved. By \eqref{item:001calculationsboring}, $\mu(f\mu(g)) \mu \in \mathcal{J}$. By \eqref{item:002calculationsboring}, $\mu(f)\mu(g) \mu \in \mathcal{J}$. Since $f\mu(\mu(g)) = \mu(f\mu(g)) - \mu(g)\mu(f)$, we have that $f \mu(\mu(g)) \mu \in \mathcal{J}$.

\end{proof}

\begin{lemma}\label{lemma:generatedidealising1}
Let $f,g \in A_{\boh}$, $\mu \in \mathcal{J}_{\boh}$. Let $I_{f,g,\mu}$ be the principal ideal of $A$ generated by $\mu(f)\mu(\mu(g))$. Then for every $a \in I_{f,g,\mu}$ and every $\tau \in \V$, $a \tau \in \mathcal{J}$.
\end{lemma}
\begin{proof}
If $a = b\mu(f)\mu(\mu(g))$, then $b\mu(\mu(g)) \mu, \ fb \mu(\mu(g)) \mu, \ \mu(f)b\mu(\mu(g)) \mu  \in \mathcal{J}$ by Lemma \ref{lemma:calculationsboring}. Therefore, if $h= b \mu(\mu(g))$
\begin{align*}
    &[b\mu(\mu(g))\mu, f \tau] - [f b \mu(\mu(g)) \mu ,\tau ]-(-1)^{|\tau| |b|} \tau(f)b\mu(\mu(g)) \mu \\
    = & h\mu(f) \tau -(-1)^{|b| |\tau|} f \tau(h) \mu + h f[\mu,\tau] +(-1)^{|\tau| |b|} \tau(f h) \mu -fh[\mu,\tau] - (-1)^{|\tau| |b|}  \tau(f) h \mu \\
    = & h \mu(f) \tau - (-1)^{|\tau| |b|} f \tau(h) \mu +(-1)^{|\tau| |b| } \tau(f) h \mu + (-1)^{|\tau| |b| }  f \tau(h) \mu - (-1)^{|\tau| |b|}  \tau(f) h \mu \\ 
    = & h \mu(f) \tau = b \mu(f)\mu(\mu(g)) \tau = a \tau \in \mathcal{J}
\end{align*}
\end{proof}

\begin{theorem}
If $X$ is a smooth integral affine supervariety and $\dim X = r|s \geq 1|0$, then $\V=\Der ( \Gamma(X,\cO) )$ is a simple Lie superalgebra.
\end{theorem}
\begin{proof}
    Let $\mathcal{J}$ be a $\Z_2$-graded ideal of $\V$ and $J = \{ a \in A \mid a \V \subset \mathcal{J} \}$. For each closed point $p \in X$ there exists $\mu \in \mathcal{J}_{\boh}$ and $f,g \in A_{\boh}$ such that $\mu(f) (p) \neq 0$ and $\mu(\mu(g)) (p) \neq 0$ by Proposition \ref{proposition:thereexistsmufwithmufpzero}. Thus, $\mu(f)\mu(\mu(g)) \neq 0$ and the principal ideal $I_{f,g,\mu}$ of $A$ generated by $\mu(f)\mu(\mu(g))$ is nonzero. By Lemma \ref{lemma:generatedidealising1}, $I_{f,g,\mu} \subset J$. Thus, $J$ is a nonzero $\V$-ideal of $A$. Furthermore, for each maximal ideal $\fm$ of $A$ there exists a $h \in J$ such that $h(p) \neq 0$ where $p \in X$ is the corresponding closed point. Hence, $h + \fm  \neq 0$, and $J$ is not contained in $\fm$. Since every proper ideal of $A$ is contained in a maximal ideal, we conclude that $J=A$. Thus, $ \V = J \V =  A\V  \subset \mathcal{J} \subset \V$. 
\end{proof}

\section*{Acknowledgments}	
The author thanks Lucas Calixto and Vyacheslav Futorny for numerous helpful discussions.
H.\,R.\,was financed by São Paulo Research Foundation (FAPESP) (grant 2020/13811-0, 2022/00184-3).


\begin{thebibliography}{}
\bibitem{BF18} Y. Billig\ and\ V. Futorny, Lie algebras of vector fields on smooth affine varieties, Comm. Algebra {\bf 46} (2018), no.~8, 3413--3429. MR3789004

\bibitem{CCF11} C. Carmeli, L. Caston\ and\ R. Fioresi, {\it Mathematical foundations of supersymmetry}, EMS Series of Lectures in Mathematics, European Mathematical Society (EMS), Z\"{u}rich, 2011. 
\bibitem{Jor86} D. A. Jordan, On the ideals of a Lie algebra of derivations, J. London Math. Soc. (2) {\bf 33} (1986), no.~1, 33--39.
\bibitem{She21} A. Sherman, Spherical supervarieties, Ann. Inst. Fourier (Grenoble) {\bf 71} (2021), no.~4, 1449--1492.
\bibitem{VMP90} A. A. Voronov, Yu. I. Manin\ and\ I. B. Penkov, Elements of supergeometry, Journal of Soviet Mathematics {\bf 51} (1990), no.~1, 2069-2083
\end{thebibliography}
\end{document}